\documentclass[12pt,oneside]{amsart}
\usepackage{amssymb, amsmath, amsthm}

\usepackage{epsfig}
\usepackage{epstopdf}
\usepackage{graphicx}
\usepackage{color}
\usepackage{enumerate}
\usepackage{amsrefs}

\usepackage[normalem]{ulem}

\usepackage[normalem]{ulem}

\usepackage{pinlabel}

\theoremstyle{plain}
\newtheorem{theorem}{Theorem}[section]

\newtheorem*{theorem*}{Theorem}
\newtheorem*{theorem-DisjtSS}{Theorem \ref{Thm: Disjt SS}}
\newtheorem*{theorem-EssentialTorus}{Theorem \ref{Thm: Essential Torus}}
\newtheorem*{cor-ScharlWu}{Corollary \ref{Cor-ScharlWu}}
\newtheorem*{corollary-MSC 2}{Corollary \ref{Cor: MSC 2}}
\newtheorem*{theorem-Main A}{Theorem \ref{Thm: Main A}}
\newtheorem*{theorem-Main B}{Theorem \ref{Thm: Main Thm B}}
\newtheorem*{Cor-Unknotting}{Theorem \ref{Cor: Prime Unknotting 1}}
\newtheorem*{Cor-genusbandsum}{Theorem \ref{Thm: Genus superadd}}
\newtheorem*{Cor-bandsumscc}{Corollary \ref{Cor: Band Sums CC}}

\newtheorem{proposition}[theorem]{Proposition}

\newtheorem{lemma}[theorem]{Lemma}
\newtheorem{definition}[theorem]{Definition}

\theoremstyle{definition}

\newcommand{\ra}{\rightarrow}
\newcommand{\x}{\times}
      \makeatletter
      \def\@setcopyright{}
      \def\serieslogo@{}
      \makeatother

\begin{document}
   % title

   \title[Idempotents in Tangle Categories Split]{Idempotents in Tangle Categories Split}
   \author{Ryan Blair and Joshua Sack}
   \email{}
   \thanks{}
   % Note that the short title for running heads goes in square
   % brackets.  This is optional.  The long title goes in curly
   % braces.  In the long title, line breaks are indicated by \\

\begin{abstract}
In this paper we use 3-manifold techniques to illuminate the structure of the category of tangles. In particular, we show that every idempotent morphism $A$ in such a category naturally splits as $A=B\circ C$ such that $C\circ B$ is an identity morphism.
\end{abstract}

\maketitle
\date{\today}

\section{Introduction}

An \emph{idempotent} of a category is a morphism that is idempotent with respect to composition, i.e.\ a morphism $f$ such that $f = f\circ f$.
Idempotents have significance to quantum observations or measurements \cite{Selinger08}, can reflect self-replication in biology (such as DNA) \cite{Kauff04}, and can form building blocks for numerous algebraic structures \cite{HaqKauffman}.
An idempotent $f$ \emph{splits} if there are morphisms $g$ and $h$ such that $f=g\circ h$, but $h\circ g$ is the identity morphism.
By direct inspection, one can see that any morphism $f$ with such a property is idempotent (if $h\circ g$ is the identity, then $(g\circ h)\circ (g\circ h) = g\circ (h\circ g)\circ h = g\circ h$); but in many categories, not all idempotents split.
A category where every idempotent splits is called \emph{Karoubi complete}.
Idempotent splitting may adopt significance from various interpretations of the categories involved.
For example, Selinger studied idempotents of dagger categories, and described in \cite[Remark 3.5]{Selinger08} how the splitting of idempotents may clarify data types. 
In \cite{Kauff04}, Kauffman related idempotents to DNA replication, and saw the idempotents in a Karoubi complete category as appealing models for self-replicators.

We show that the category of unoriented tangles up to isotopy is Karoubi complete.
Objects of this category are points in the disc $D^2$, morphisms are properly embedded 1-manifolds in $D^2\times I$ (these are the tangles), and the morphism composition is achieved via a stacking operation.
Categories of tangles were studied in \cite{Yetter} to understand the combinatorial structure of tangle composition, and various categories of tangles are classified in \cite{Freyd89} as certain types of braided pivotal categories.
A related category, called the Temperley-Lieb category, can be described similarly, but with $D^2$ replaced by $I$, and hence the category of tangles we consider is a natural generalization of the Temperley-Lieb category.
It was shown in \cite{Abramsky07} that the Temperley-Lieb category is Karoubi complete.

Although the main result of this paper extends that in \cite{Abramsky07}, the techniques in this paper are different and are rather inspired by the proof of the prime decomposition theorem for string links provided in \cite{BBK}. The main technical tool in the current paper, as well as in \cite{BBK}, is a bound, established in \cite{FF}, on the number of non-parallel essential surfaces in a compact 3-manifold. The paper is organized as follows. In Section \ref{defs} we precisely define the tangle category. In Section \ref{incompressible} we review incompressible punctured surfaces and their properties. In Section \ref{braid} we adapt to tangles the notion of braid-equivalence for string links in \cite{BBK} and apply this idea to factoring morphisms as the composition of two morphisms. In Section \ref{main} we prove that all idempotents in the category of tangles split.

\section{Acknowledgments}

The authors would like to thank Michael Peterson for many useful conversations.

\section{Idempotents in the Category of Tangles}\label{defs}

Let $\mathcal{T}$ be the category of tangles. The definition of $\mathcal{T}$ that we give here is essentially equivalent to the definition of the \emph{category of unoriented tangles up to isotopy}, denoted $\mathbb{TANG}$, in \cite{Freyd89}. The objects of $\mathcal{T}$ are the natural numbers. Each natural number $n$ is identified with $n$ distinct fixed points $\{x_1, x_2,...,x_n\}$ in $D^2$. The morphisms of $\mathcal{T}$ are \emph{tangles}. A tangle is a pair $(D^2 \times I,A)$ such that $A$ is a properly embedded compact 1-manifold in $D^2 \times I$ with the following conditions: the boundary of each arc (connected component with non-trivial boundary) of $A$ is contained in $(D^2\times \{0\})\cup (D^2\times \{1\})$; the intersection of $A$ with the lower and upper boundaries of $D^2\times I$ are $A\cap (D^2\times \{0\})=\{(x_1,0),(x_2,0),...,(x_n,0)\}$ and $A\cap (D^2\times \{1\})=\{(x_1,1),(x_2,1),...,(x_m,1)\}$; and for each boundary point $(x_i,0)$ or $(x_i,1)$ of $A$ and each sufficiently small neighborhood of that point, the derivatives of all orders of the embedding agree with the maps $t \mapsto (x_i, t)$. For simplicity, we will occasionally refer to the tangle $(D^2 \times I,A)$ as the morphism $A$. We denote $D^2 \times \{1\}$ by $\partial_+ (D^2 \times I)$  and $D^2 \times \{0\}$ by $\partial_- (D^2 \times I)$.

\begin{definition}[Tangle equivalence]
\label{StringLinkEquivalence}
Tangles $(D^2 \times I,A)$ and $(D^2 \times I,B)$ are equivalent if there is an isotopy of $D^2 \times I$ fixing $\partial (D^2 \times I)$ that takes $A$ to $B$.  In this case, (by abuse of notation) we will write $A = B$.
\end{definition}

%The category that we define above is equivalent to the \emph{category of unoriented tangles up to isotopy}, denoted $\mathbb{TANG}$, in \cite{Freyd89}.

Let $h:D^2 \times I \ra I$ be the projection map onto the second coordinate. A \emph{braid} is a tangle that is equivalent to a tangle $(D^2 \times I, A)$ with the property that for every component $\alpha$ of $A$, the restriction of $h$ to $\alpha$ is a smooth, one-to-one and onto function with no critical points in its domain. Given tangle $(M,A)$ from $k$ to $l$ and tangle $(N,B)$ from $l$ to $m$, with $M=N=D^2\times I$, denote the \emph{composition} of these morphisms by $(D^2 \times I, A\circ B)$ which is the quotient of $M \cup N$ achieved by gluing $\partial_+ (M)$ to $\partial_- (N)$ via the map $(x,1)\mapsto (x,0)$, and $A\circ B$ is the properly embedded 1-manifold in the quotient which is the image of $A \cup B$ under this identification. The resulting quotient of $M \cup N$ is again homeomorphic to $D^2 \times I$ and we choose to identify the image of $M$ under the quotient map with $D^2 \x [0,1/2]$ and the image of $N$ with $D^2 \x [1/2,1]$ in the obvious ways. Again, by abuse of notation, we will write the resulting tangle as the morphism $A\circ B$ and consider it the composition of morphisms $A$ and $B$. See Figure \ref{StringlinkStacking}.

%achieved by gluing $\partial_+ (M)$ to $\partial_- (N)$ via the identity map and considering the properly embedded 1-manifold which is the image of $A \cup B$ under this identification. The resulting quotient of $M \cup N$ is again homeomorphic to $D^2 \times I$ and we choose to identify the image of $M$ under the quotient map with $D^2 \x [0,1/2]$ and the image of $N$ with $D^2 \x [1/2,1]$ in the obvious ways. Again, by abuse of notation, we will write the resulting quotient as morphism $A\circ B$ and consider it the composition of morphisms $A$ and $B$. See Figure \ref{StringlinkStacking}.

\begin{figure}[h!]
\begin{picture}(250,315)
\put(11,11){\includegraphics[scale=.5]{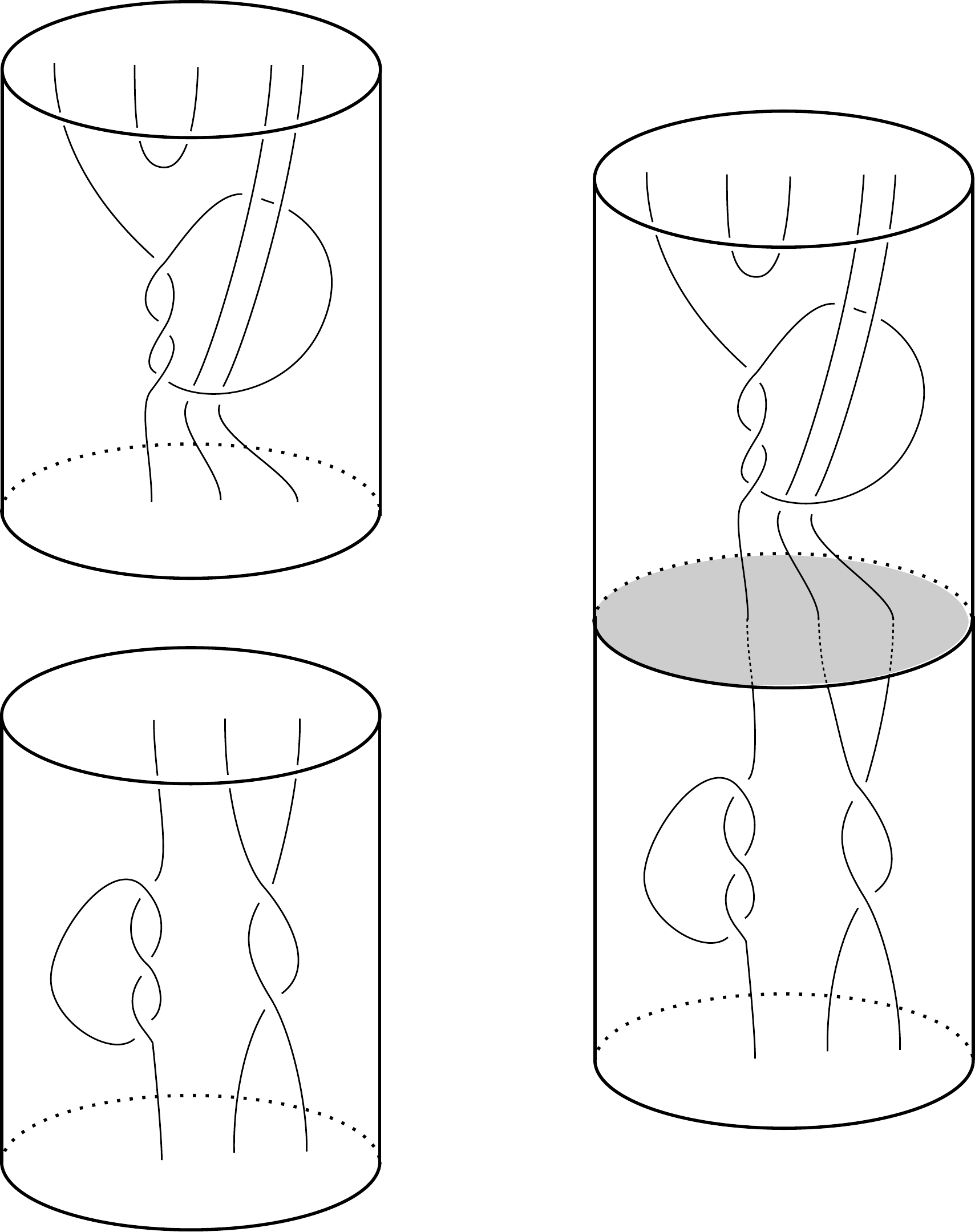}}
\put(0,80){$A$}
\put(0, 240){$B$}
\put(195, 25){$A \circ B$}
\put(145,160){$F$}
\put(55,0){$\partial_-$}
\put(55,163){$\partial_+$}
\end{picture}
\caption{The composition of tangles}
\label{StringlinkStacking}
\end{figure}
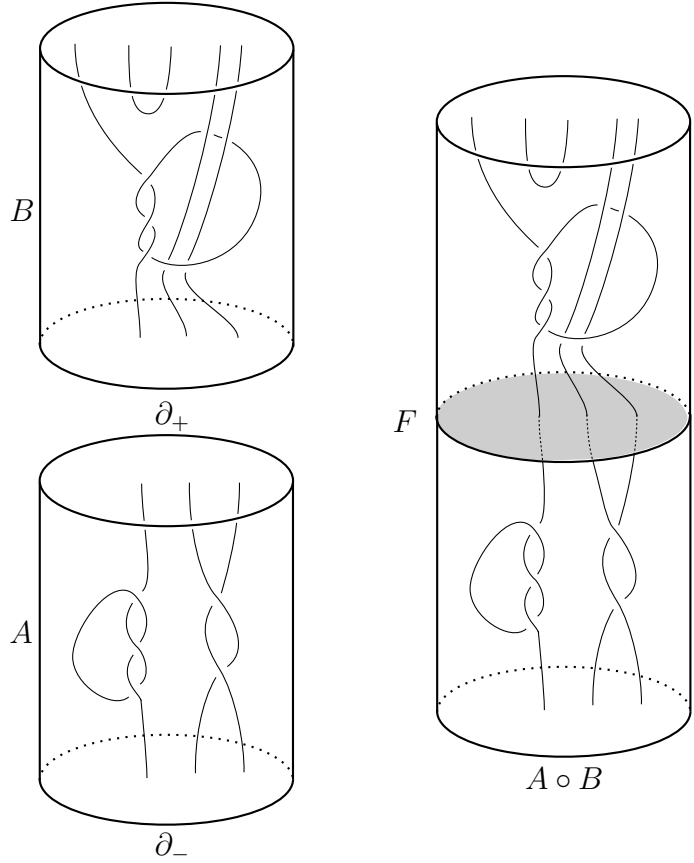

We illustrate in Figure \ref{idempotent} an \emph{idempotent} in the category of tangles, i.e.\ a morphism $A$ of $\mathcal{T}$ such that $A\circ A = A$.
Note that if $A$ is a braid and an idempotent then $A$ is an identity morphism, since braids form a group under composition. 
Recall that a category is \emph{Karoubi complete} if all of its idempotents split, where an idempotent $A$ \emph{splits} if there exist morphisms $C$ and $B$ such that $A=C\circ B$ and $B\circ C = \mathit{Id}_n$ is an identity morphism. Our main theorem is the following.
 
 \begin{theorem}\label{thm:main}
 The category $\mathcal{T}$ of tangles is Karoubi complete.
% All idempotent morphisms in $\mathcal{T}$ split.
 \end{theorem}

 \begin{figure}[h!]
\begin{picture}(230,280)
\put(0,0){\includegraphics[scale=.5]{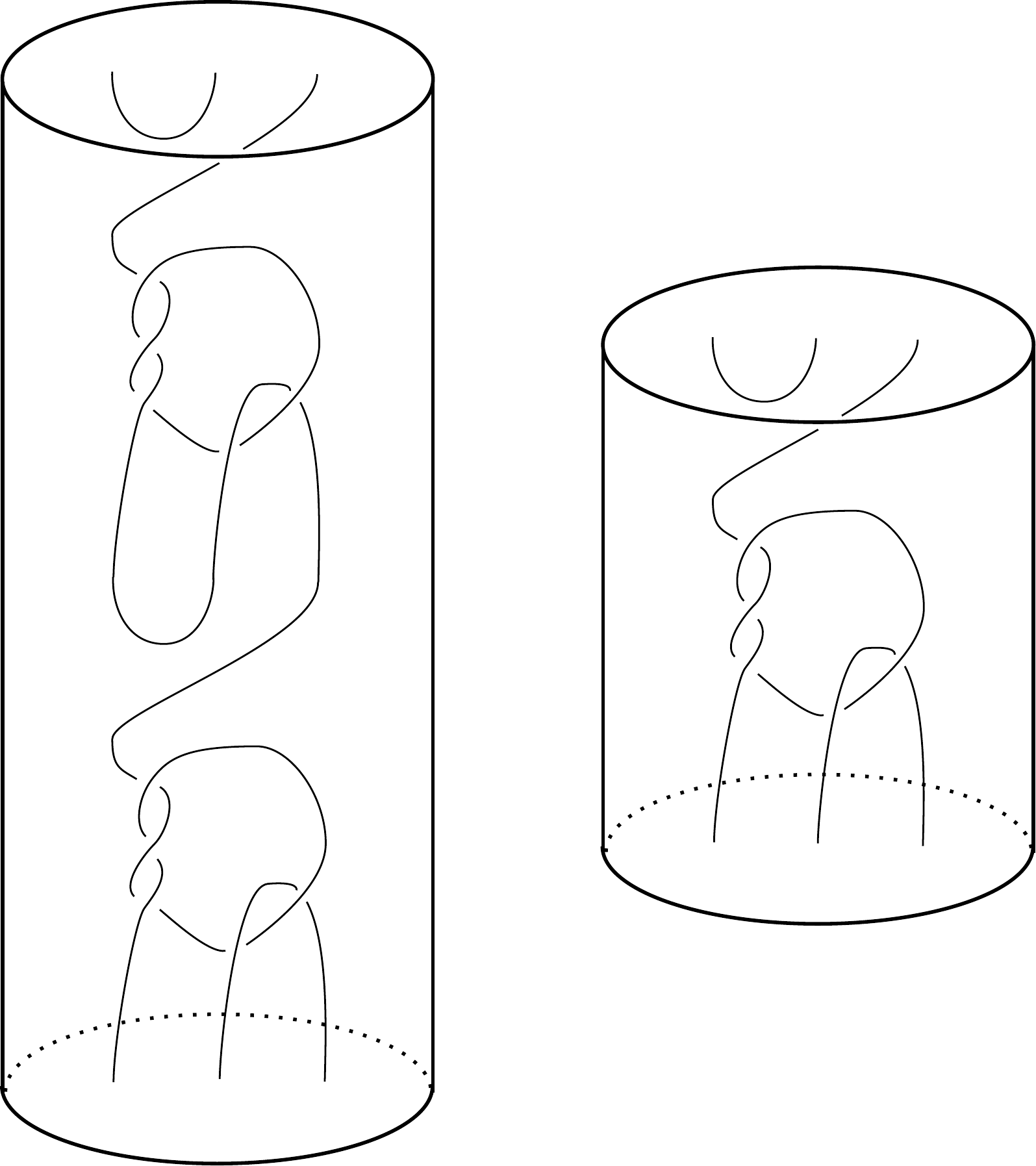}}
\put(113,125){$\cong$}
%\put(0, 240){$B$}
%\put(195, 25){$A \circ B$}
%\put(145,160){$F$}
%\put(55,0){$\partial_-$}
%\put(55,163){$\partial_+$}
\end{picture}
\caption{An idempotent morphism in the category of tangles}
\label{idempotent}
\end{figure}

\section{Incompressible punctured surfaces}\label{incompressible}

Our primary tool in the classification of idempotents in $\mathcal{T}$ will be the study of punctured surfaces up to transverse isotopy. Unless otherwise stated all manifolds are compact. Suppose $\alpha$ is a 1-manifold properly embedded in a 3-manifold $M$. If $F$ is a properly embedded surface in $M$ which meets $\alpha$ transversely in $k$ points, we say $F$ is \emph{$k$-punctured}. An isotopy $\phi_t$ of $F$ in $M$ is \emph{proper} if its restriction to $\partial F\times I$ is an isotopy of $\partial F$ in $\partial M$. 
Moreover, the isotopy $\phi_t$ is \emph{transverse} to $\alpha$ if the embedding $\phi_t$ is transverse to $\alpha$ for all fixed values of $t$. Isotopies of surfaces in this paper will always be proper isotopies that are transverse to the relevant 1-manifolds.

If $\alpha$ is a 1-manifold properly embedded in $M\cong D^2 \times I$ and $F$ is a properly embedded $k$-punctured surface in $M$, $F$ is \emph{boundary-parallel} either if $F$ is a $2$-punctured $2$-sphere bounding a $3$-ball that meets $\alpha$ in an unknotted arc or if there is a proper, transverse isotopy of $F$ in $M$ which takes $F$ to a punctured subsurface contained in $\partial M$. Otherwise, we say $F$ is \emph{non-boundary parallel}. A loop $\gamma$ embedded in $F$ is \emph{essential} if it does not bound a 0-punctured disk in $F$. The $k$-punctured surface $F$ is \emph{compressible} in $(M,\alpha)$ (or just \emph{compressible} when context is understood)  if $F$ is a $0$-punctured $2$-sphere bounding a $3$-ball or if
there exists a disk $D$ embedded in $M$ such that $D \cap F=\partial D$, $\partial D$ is essential in $F$ and $D$ is disjoint from $\alpha$. Such a disk is called a \emph{compressing disk}.  Otherwise, we say $F$ is \emph{incompressible}. 
A punctured surface $F$ is \emph{essential} if $F$ is incompressible and non-boundary parallel.

Given a 1-manifold $\alpha$ properly embedded in $M\cong D^2 \times I$ and $F$ a properly embedded $k$-punctured surface with compressing disk $D$, we can \emph{compress} $F$ along $D$ to form a new embedded $k$-punctured surface $F^*$. See Figure \ref{compress.fig}. Let $D^2\times I$ be a fibered neighborhood of $D$ in $M$ such that $D=D^2\times \{\frac{1}{2}\}$, $D^2\times I$ is disjoint from $\alpha$, and $\partial(D^2) \times I$ is an embedded annulus in $F$ that is disjoint from the punctures of $F$. Then we define $F^*$ to be a surface transversely isotopic to $(F\setminus (\partial(D^2) \times I)) \cup (D^2\times \{0,1\})$.

\begin{figure}
	\includegraphics[width=4in]{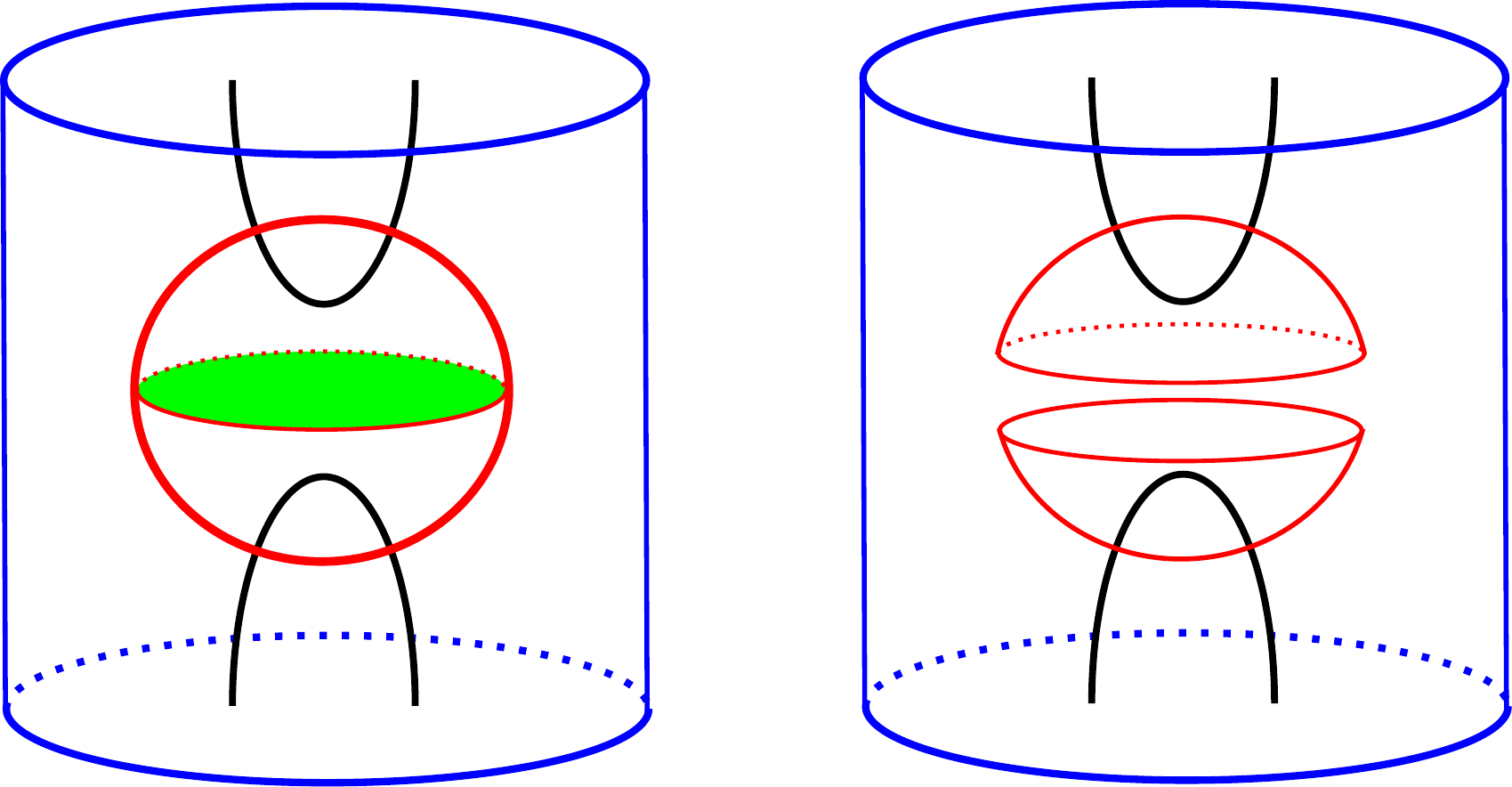}
	\caption{An example of compressing a 4-punctured sphere to obtain two boundary-parallel 2-punctured spheres.}
	\label{compress.fig}
	\end{figure}

Although we take the point of view of incompressible and non-boundary parallel punctured surfaces in this paper, we could have equivalently adopted the point of view of studying incompressible and non-boundary parallel meridional surfaces properly embedded in the exterior of $\alpha$ in $M$. In particular, if $F$ is a properly embedded non-boundary parallel punctured surface in $(D^2\times I, \alpha)$ and $\eta(\alpha)$ is a small open regular neighborhood of $\alpha$ in $D^2\times I$, then $F\setminus \eta(\alpha)$ is non-boundary parallel in $(D^2\times I)\setminus \eta(\alpha)$. Similarly, if $F$ is a properly embedded incompressible punctured surface in $(D^2\times I, \alpha)$ and $\eta(\alpha)$ is a small open regular neighborhood of $\alpha$ in $D^2\times I$, then $F\setminus \eta(\alpha)$ is incompressible in $(D^2\times I)\setminus \eta(\alpha)$. In particular, we will make use of the following Theorem of Freedman and Freedman.

\begin{theorem}\label{Freedman}~\cite{FF} Let $M$ be a compact 3-manifold with boundary and $b$ an integer greater than
zero. There is a constant $c(M, b)$ so that if $F_1, . . . , F_k$, $k > c$, is a collection of incompressible
surfaces such that all the Betti numbers $b_{1}(F_{i}) < b$, $1 \leq i \leq k$, and no $F_i$, $1 \leq i \leq k$, is a boundary
parallel annulus or a boundary parallel disk, then at least two members $F_i$ and $F_j$ are parallel.
\end{theorem}

Note that Freedman and Freedman define two disjoint properly embedded surfaces $F_i$ and $F_j$ in a compact 3-manifold $M$ to be \emph{parallel} if $F_i\cup F_j$ cobound a product $F\times I$ in $M$ such that $\partial F\times I \subset \partial M$, $F_i=F\times \{0\}$ and $F_j=F\times \{1\}$.

\section{Decomposing Disks and Braid Equivalence}\label{braid}
Decomposing a morphism in $\mathcal{T}$ as a composition of two morphisms involves some amount of choice. This choice can be captured via the notion of braid-equivalence.

\begin{definition}
\label{BraidEquivalenceDef}
Two tangles $(D^2 \times I, A)$ and $(D^2 \times I, B)$  are \emph{braid-equivalent} if there exist braids $C_1$ and $C_2$ such that $A=C_{1}\circ B\circ C_{2}$.
\end{definition}

\begin{proposition}
\label{BraidEquivalenceProp}
Tangles  $(D^2 \times I, T_1)$ and $(D^2 \times I, T_2)$ are braid-equivalent if and only if there is an isotopy of $D^2 \times I$ which fixes $(\partial D^2) \x I$ and which takes $T_1$ to $T_2$.
\end{proposition}

\begin{proof} 
This follows from a nearly identical adaptation of the proof of Proposition 3.6 of \cite{BBK}.
\end{proof}

\begin{definition}
A \emph{decomposing disk} for a tangle $(D^2 \times I, A)$ is a punctured disk which is properly embedded in $D^2 \times I$,
whose boundary is isotopic in $\partial(D^2 \x I)$ to $\partial (\partial_+ (D^2 \times I))$. See the disk $F$ in Figure \ref{StringlinkStacking}.
\end{definition}

A decomposing disk $F$ for a tangle $(D^2 \times I, A)$ separates
$D^2 \times I$ into two connected components, one containing $\partial_- (D^2 \times I)$ and the other containing $\partial_+ (D^2 \times I)$.  The closure of each component is homeomorphic to $D^2 \x I$, so $F$ decomposes $(D^2 \times I, A)$ into two tangles, each of which is well-defined up to composition with braids (cf.~\emph{braid-equivalence} in Definition \ref{BraidEquivalenceDef}).  If $(D^2 \times I, B)$ is the tangle resulting from restricting $A$ to the side of $F$ in $D^2 \times I$ that contains $\partial_- (D^2 \times I)$ and $(D^2 \times I, C)$ is the tangle resulting from restricting $A$ to the side of $F$ in $D^2 \times I$ that contains $\partial_+ (D^2 \times I)$, we say that $F$ \emph{decomposes $(D^2 \times I, A)$ as $(D^2 \times I, B\circ C)$}, or, more simply, $A=B\circ C$.

\section{The Proof}\label{main}

\begin{definition}
A decomposing disk $F$ for a tangle $(D^2 \times I, A)$ is \emph{minimal} if there is no decomposing disk $G$ such that $|F\cap A|>|G\cap A|$. See Figure \ref{idempotent.fig}.
\end{definition}

\begin{figure}
	\includegraphics[width=1in]{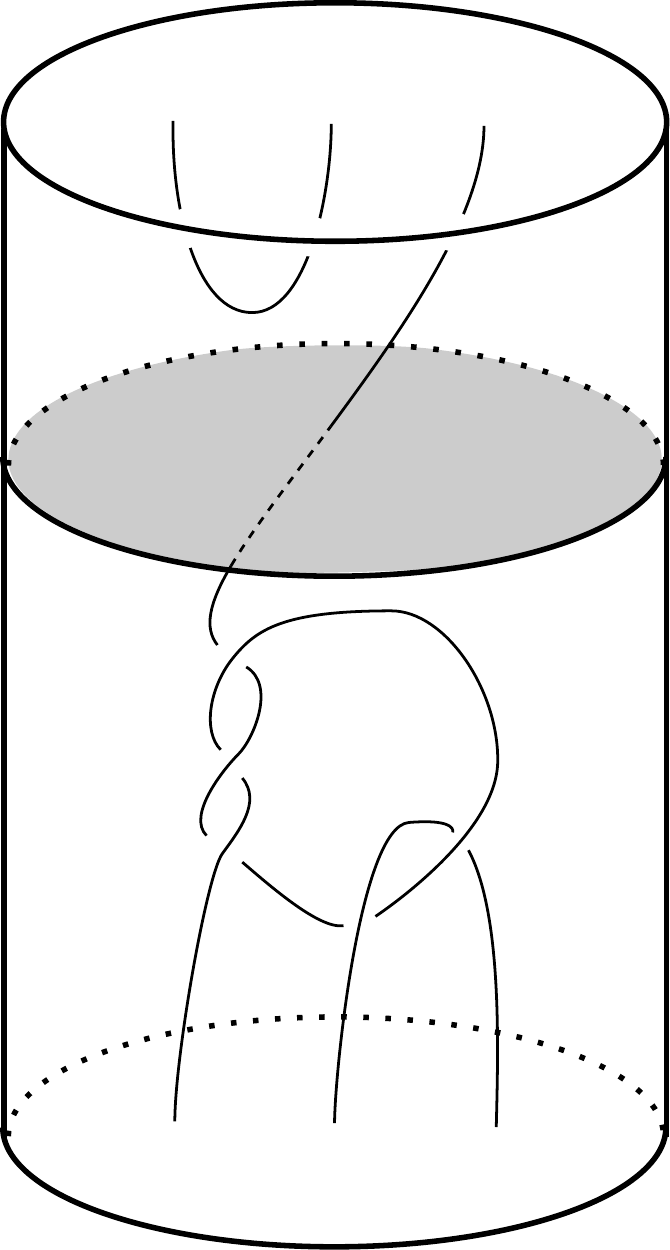}
	\caption{A minimal decomposing disk for an idempotent tangle.}
	\label{idempotent.fig}
	\end{figure}

%\begin{definition}
%A tangle $(D^2 \times I, A)$ is an \emph{idempotent} if $(D^2 \times I, A)$ is equivalent to $(D^2 \times I, A\circ A)$.
%\end{definition}

%\begin{remark}
%Since braids on $n$-strands form a group under composition, then all idempotent braids are identity morphisms.
%\end{remark}

\begin{lemma}\label{ess}
If $(D^2 \times I, A)$ is an idempotent, then either $(D^2 \times I, A)$ is an identity morphism or any minimal decomposing disk is essential.
\end{lemma}

\begin{proof}
Assume that $(D^2 \times I, A)$ is an idempotent. If $(D^2 \times I, A)$ is a braid, then, as braids form a group, $A$ is an identity morphism. So, we may assume that $(D^2 \times I, A)$ is not a braid. If $A$ contains $l$ distinct closed loops, then, since $A$ is idempotent, $A$ must contain $2l$ distinct closed loops, a contradiction unless $l=0$. Hence, we can assume that $A$ contains no closed loops.

Since $(D^2 \times I, A)$ is an idempotent, then $(D^2 \times I, A)$ is a morphism from $n$ points to $n$ points for some $n$. Since $A$ is non-empty, $n\geq 1$. By Theorem \ref{Freedman}, there is an integer $c$ such that if $F_1, . . . , F_k$ is a collection of disjoint incompressible decomposing disks in $D^2 \times I$ that each meet $A$ in at most $n$ points (hence the first Betti number of each $F_i$ is bounded above by $n$) and $k>c$, then at least two members $F_i$ and $F_j$ are parallel  or one of the surfaces $F_1, . . . , F_k$ is a boundary parallel 0-punctured disk or a boundary parallel 1-punctured disk.

\emph{Claim}: $\partial_+ (D^2 \times I)$ or $\partial_- (D^2 \times I)$ is compressible in $(D^2 \times I, A)$.\footnote{Observe that in Figure~\ref{idempotent} $\partial_+(D^2\times I)$ is compresible.}

\emph{Proof of claim}: Since $(D^2 \times I, A)$ is equivalent to $(D^2 \times I, A\circ A)$, then $(D^2 \times I, A)$ is equivalent to $(D^2 \times I, A^{c+2})$. Hence, we can find $c+1$ pairwise decomposing disks, $F_1, . . . , F_{c+1}$, for $(D^2 \times I, A)$ that decompose $(D^2 \times I, A)$ into $c+2$ copies of $(D^2 \times I, A)$. If both $\partial_+ (D^2 \times I)$ and $\partial_- (D^2 \times I)$ are incompressible in $(D^2 \times I, A)$, then each of $F_1, . . . , F_{c+1}$ are incompressible in $(D^2 \times I, A)$. Since $n\geq 1$, then none of the surfaces $F_1, . . . , F_{c+1}$ is a 0-punctured disk. Moreover, if any of the surfaces was a boundary parallel once punctured disk (i.e. a boundary parallel annulus in the exterior of $A$), then, by the isotopy extension theorem \cite{Palais60}, $A$ would be a trivial braid on one strand. Thus, the collection $F_1, . . . , F_{c+1}$ meets the hypothesis of Theorem \ref{Freedman} and there exist two members $F_i$ and $F_j$ that are parallel. The tangle between $F_i$ and $F_j$ in $D^2 \times I$ is braid-equivalent to $(D^2 \times I, A^l)$ for some $l\geq 1$; however, since $F_i$ is parallel to $F_j$, then $A^l$ and, thus, $A$ is a braid, a contradiction. Hence, one of $\partial_+ (D^2 \times I)$ or $\partial_- (D^2 \times I)$ is compressible in $(D^2 \times I, A)$. $\square$

Without loss of generality, suppose that $\partial_+ (D^2 \times I)$ is compressible in $(D^2 \times I, A)$. Compressing $\partial_+ (D^2 \times I)$ once results in a surface with two connected components. One component is a decomposing disk that meets $A$ in strictly fewer than $n$ points and the other component is a punctured sphere. 
Note that any boundary parallel decomposing disk would be properly, transversely isotopic to $\partial_+ (D^2 \times I)$ or $\partial_- (D^2 \times I)$ in $(D^2 \times I, A)$, and hence meets $A$ in $n$ points.
Since we have found a decomposing disk that meets $A$ in strictly fewer than $n$ points, then a minimal decomposing disk cannot be boundary parallel. 

Let $F$ be a minimal decomposing disk for $(D^2 \times I, A)$.  By the above argument, $F$ is non-boundary parallel. Next, we show that $F$ is incompressible. If $F$ is compressible, then compressing $F$ once results in a surface with two connected components. One component is a decomposing disk that meets $A$ in fewer points than $F$ does. This is a contradiction to $F$ being minimal. Hence, $F$ is incompressible. Since $F$ is both incompressible and non-boundary parallel, then $F$ is essential.
\end{proof}

We now restate and prove our main theorem (Theorem~\ref{thm:main}).
\begin{theorem}
If $(D^2 \times I, A)$ is an idempotent, then there exist $B$ and $C$, such that $A=B\circ C$ and $C\circ B$ is an identity morphism.
\end{theorem}

\begin{proof}

Let $F$ be a minimal decomposing disk for $(D^2 \times I, A)$. By Lemma \ref{ess}, $F$ is essential. Denote the tangles that $F$ decomposes $(D^2 \times I, A)$ into by $(D^2 \times I, B)$ and $(D^2 \times I, C)$ so that $A=B\circ C$. Note that since $A$ is an idempotent, then $A$ is a morphism from $n$ points to $n$ points for some $n$. Moreover, since $\partial_+ (D^2\times I)$ is a decomposing disk for every $(D^2 \times I, A)$, then $F$ meets $A$ in at most $n$ points.

By Theorem \ref{Freedman}, there is an integer $c$ such that if $F_1, . . . , F_k$ is a collection of disjoint essential decomposing disks in $D^2 \times I$ that each meet $A$ in at most $n$ points (hence the first Betti number of each $F_i$ is bounded above by $n$) and $k>c$, then at least two members $F_i$ and $F_j$ are parallel (note that none of the $F_i$ are boundary parallel since each is essential). 

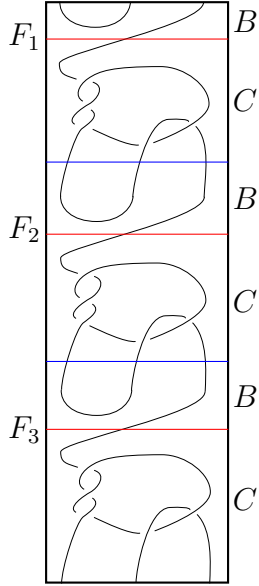
\begin{figure}[h!]
\begin{picture}(90,250)
\put(14,0){\includegraphics[scale=.5]{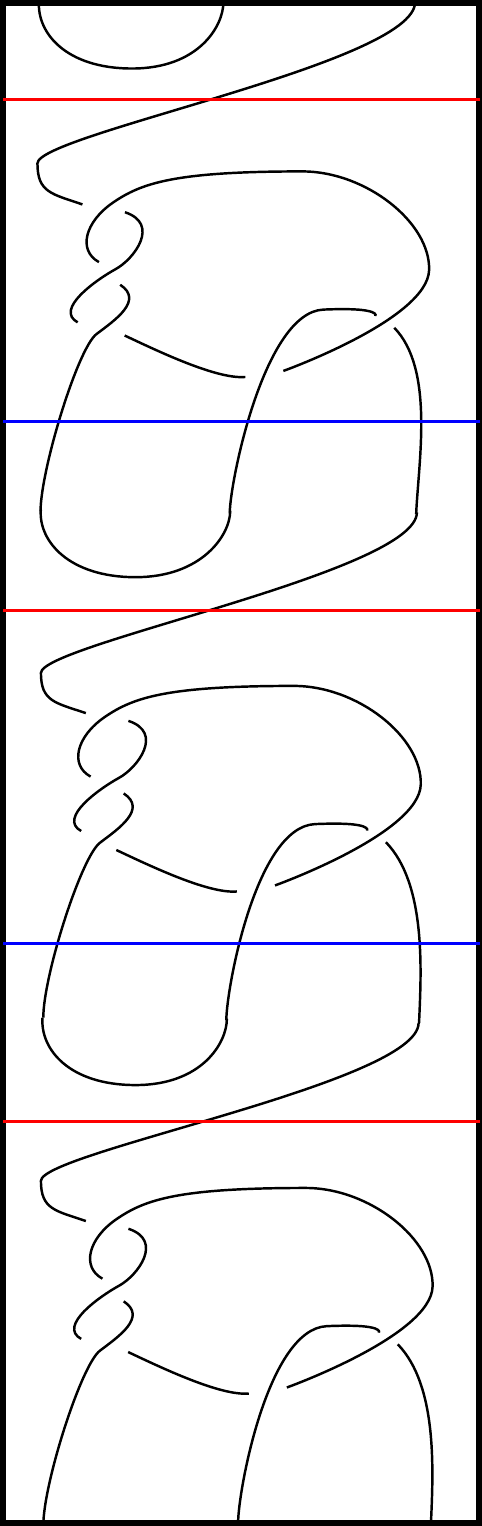}}
\put(0,204){$F_1$}
\put(0, 131){$F_2$}
\put(0, 55){$F_3$}
\put(85,209){$B$}
\put(85, 142){$B$}
\put(85, 67){$B$}
\put(85,179){$C$}
\put(85, 103){$C$}
\put(85, 28){$C$}
%\put(145,160){$F$}
%\put(55,0){$\partial_-$}
%\put(55,163){$\partial_+$}
\end{picture}
\caption{An example of how three copies of an idempotent $A=B\circ C$ can be decomposed into one copy of $B$, two copies of $C\circ B$, and one copy of $C$.}
\label{pattern}
\end{figure}

%or one of the surfaces $F_1, . . . , F_k$ is a 0-punctured disk or a boundary parallel annulus.

Since we have established that $F$ is an essential punctured surface in $(D^2 \times I, A)$ and since $(D^2 \times I, A)$ is equivalent to $(D^2 \times I, A^{c+1})$, then we can find $c+1$ disjoint minimal decomposing disks, $F_1, . . . , F_{c+1}$, for $(D^2 \times I, A)$ each representing the copy of $F$ in each copy of $A$ in $A^{c+1}$.  Each of $F_1, . . . , F_{c+1}$ is essential in $(D^2 \times I, A)$ and together they decompose $(D^2 \times I, A)$ into one copy of $(D^2 \times I, B)$, $c$ copies of $(D^2 \times I, C\circ B)$, and one copy of $(D^2 \times I, C)$. See Figure \ref{pattern}. By Theorem \ref{Freedman}, there exist two members $F_i$ and $F_j$ that are parallel. The tangle between $F_i$ and $F_j$ in $D^2 \times I$ is equivalent to $(D^2 \times I, (C\circ B)^l)$ for some $l\geq 1$, however, since $F_i$ is parallel to $F_j$, then $(C\circ B)^l$ and, thus, $C\circ B$ is a braid.

Since $A^2=A$, then $B\circ C\circ B \circ C=B \circ C$. We can compose on the left by $C$ and the right by $B$ to obtain $(C\circ B)^3=(C\circ B)^2$. However, since braids on $n$ strands form a group under composition, $(C\circ B)^3=(C\circ B)^2$ implies $C\circ B$ is an identity morphism.
\end{proof}

\bibliographystyle{plain}
\bibliography{bib}

\begin{bibdiv}
\begin{biblist}

\bib{Abramsky07}{incollection}{
      author={Abramsky, Samson},
       title={Temperley-lieb algebra: from knot theory to logic and computation
  via quantum mechanics},
        date={2007},
   booktitle={Mathematics of quantum computation and quantum technology},
   publisher={Taylor and Francis},
       pages={515\ndash 558},
}

\bib{BBK}{article}{
      author={Blair, Ryan},
      author={Burke, John},
      author={Koytcheff, Robin},
       title={A prime decomposition theorem for the 2-string link monoid},
        date={2015},
        ISSN={0218-2165},
     journal={J. Knot Theory Ramifications},
      volume={24},
      number={2},
       pages={1550005, 24},
         url={http://dx.doi.org/10.1142/S0218216515500054},
      review={\MR{3334658}},
}

\bib{FF}{article}{
      author={Freedman, Benedict},
      author={Freedman, Michael~H.},
       title={Kneser-{H}aken finiteness for bounded {$3$}-manifolds locally
  free groups, and cyclic covers},
        date={1998},
        ISSN={0040-9383},
     journal={Topology},
      volume={37},
      number={1},
       pages={133\ndash 147},
         url={http://dx.doi.org/10.1016/S0040-9383(97)00007-4},
      review={\MR{1480882 (99h:57036)}},
}

\bib{Freyd89}{article}{
      author={Freyd, Peter~J},
      author={Yetter, David~N},
       title={Braided compact closed categories with applications to low
  dimensional topology},
        date={1989},
        ISSN={0001-8708},
     journal={Advances in Mathematics},
      volume={77},
      number={2},
       pages={156 \ndash  182},
  url={http://www.sciencedirect.com/science/article/pii/0001870889900182},
}

\bib{HaqKauffman}{article}{
      author={Haq, Rukhsan~Ul},
      author={Kauffman, Louis~H.},
       title={Iterants, idempotents and clifford algebra in quantum theory},
        date={2017},
     journal={CoRR},
      volume={abs/1705.06600},
         url={https://arxiv.org/abs/1705.06600},
}

\bib{Kauff04}{incollection}{
      author={Kauffman, Louis~H.},
       title={Biologic. {II}},
        date={2004},
   booktitle={Woods {H}ole mathematics},
      series={Ser. Knots Everything},
      volume={34},
   publisher={World Sci. Publ., Hackensack, NJ},
       pages={94\ndash 132},
         url={http://dx.doi.org/10.1142/9789812701398_0003},
      review={\MR{2123368}},
}

\bib{Palais60}{article}{
      author={Palais, Richard~S.},
       title={Local triviality of the restriction map for embeddings},
        date={1960},
        ISSN={0010-2571},
     journal={Comment. Math. Helv.},
      volume={34},
       pages={305\ndash 312},
         url={http://dx.doi.org/10.1007/BF02565942},
      review={\MR{0123338}},
}

\bib{Selinger08}{article}{
      author={Selinger, Peter},
       title={Idempotents in dagger categories},
        date={2008},
        ISSN={1571-0661},
     journal={Electronic Notes in Theoretical Computer Science},
      volume={210},
       pages={107 \ndash  122},
  url={http://www.sciencedirect.com/science/article/pii/S1571066108002363},
        note={Proceedings of the 4th International Workshop on Quantum
  Programming Languages (QPL 2006)},
}

\bib{Yetter}{article}{
      author={Yetter, David~N.},
       title={Markov algebras},
        date={1988},
     journal={Contemporary Mathematics},
      volume={78},
}

\end{biblist}
\end{bibdiv}

\end{document}